\documentclass[11 pt,english]{article}
\usepackage[activeacute]{babel}
\usepackage{amsfonts} 
\usepackage{mathrsfs} 
\usepackage{amssymb, amsmath, amsfonts}
\usepackage{makeidx}
\usepackage{epsfig}
\usepackage{color}
\usepackage[T1]{fontenc}
\usepackage{graphics,graphicx,graphpap}
\usepackage[ansinew]{inputenc}
\usepackage{amsthm}
\usepackage{multirow}

\newcommand\blfootnote[1]{%
  \begingroup
  \renewcommand\thefootnote{}\footnote{#1}%
  \addtocounter{footnote}{-1}%
  \endgroup
}

\if@twoside \oddsidemargin 5pt \evensidemargin 5pt \marginparwidth 20pt
 \else \oddsidemargin 5pt \evensidemargin 5pt \marginparwidth 20pt \fi

\newcommand{\ZZ}{\mathbb{Z}}

\newcommand{\G}{\Gamma}

\newcommand{\la}{\langle}
\newcommand{\ra}{\rangle}

\newcommand{\uB}{\mathrm{uB}}
\newcommand{\auB}{\overline{\mathrm{uB}}}
\renewcommand{\SS}[2]{S_{#1}\hspace{-2pt}\cdot\hspace{-2pt} S_{#2}}
\newcommand{\SSx}[2]{S_{#1}\hspace{-3pt}-\hspace{-3pt}S_{#2}}
\newcommand{\Ki}{\mathrm{Ki}}
\newcommand{\Sp}{\mathrm{Sp}}

\newtheorem{theorem}{Theorem}[section]
\newtheorem{proposition}[theorem]{Proposition}

\newtheorem{problem}[theorem]{Problem}
\newtheorem{conjecture}[theorem]{Conjecture}

\theoremstyle{definition}

\newtheorem{construction}[theorem]{Construction}

\begin{document}

\begin{center}
\Large{\textbf{Distance-unbalancedness of graphs}} \\ [+4ex]
\v Stefko Miklavi\v c{\small$^{a, b, c,*}$}, Primo\v z \v Sparl{\small$^{a, c, d}$}
\\ [+2ex]
{\it \small 
$^a$University of Primorska, Institute Andrej Maru\v si\v c, Koper, Slovenia\\
$^b$University of Primorska, FAMNIT, Koper, Slovenia\\
$^c$Institute of Mathematics, Physics and Mechanics, Ljubljana, Slovenia\\ 
$^d$University of Ljubljana, Faculty of Education, Ljubljana, Slovenia\\
}
\end{center}

%\footnotetext{$^*$-corresponding author\\
\blfootnote{
Email addresses: 
stefko.miklavic@upr.si (\v Stefko Miklavi\v c),
primoz.sparl@pef.uni-lj.si (Primo\v z \v Sparl)
\\
* - corresponding author
}

%%%%%%%%%%%%%%%%%
%%%		Abstract
%%%%%%%%%%%%%%%%%

\hrule

\begin{abstract}
In this paper we propose and study a new structural invariant for graphs, called distance-unbalanced\-ness, as a measure of how much a graph is (un)balanced in terms of distances. Explicit formulas are presented for several classes of well-known graphs. Distance-unbalancedness of trees is also studied. A few conjectures are stated and some open problems are proposed.
\end{abstract}

\hrule

\begin{quotation}
\end{quotation}

%%%%%%%%%%%%%%%%%
\section{Introduction}
\label{sec:Intro}

In mathematical chemistry as well as in graph theory there is a great interest in the study of various graph parameters that are defined as sums of certain quantities over all vertices (such as the first Zagreb index and the Graovac-Pisanski index), over all pairs of adjacent vertices (such as the Szeged index, the second Zagreb index, the Balaban index, the Randi\'c index, the Mostar index and the irregularity of a graph) or over all pairs of vertices of the corresponding graph (such as the Wiener index and the total irregularity of a graph). See, for example, \cite{BDHKS, AD, DMSTZ18, GFK, KST, KST1, WB} for the most recent research on these indices. In many of these parameters the above mentioned quantities are distance-related. For instance, in the Wiener index the corresponding quantity is the distance between the two vertices in question while in the Szeged index the corresponding quantity for the pair $\{u,v\}$ is the product $n_{u,v}n_{v,u}$, where for any two vertices $x$ and $y$ of the graph in question $n_{x,y}$ is the number of vertices of the graph that are closer to $x$ than to $y$. 

The most recently introduced index of the above mentioned ones is the {\em Mostar index}~\cite{DMSTZ18} which is defined as
$$
	\mathrm{Mo}(\G) = \sum_{\{u,v\} \in E(\G)}|n_{u,v} - n_{v,u}|.
$$
This index is intimately related to the concept of distance-balancedness of graphs, which was first studied in~\cite{Han99} although the term {\em distance-balanced} was introduced nine years later~\cite{JKR08}. In terms of the above mentioned Mostar index a graph is distance-balanced if and only if its Mostar index is equal to $0$.

In~\cite{MikSpa18} the following related concept was introduced. For a pair of (not necessarily adjacent) vertices $u$ and $v$ of a (connected) graph we say that this pair of vertices is {\em balanced} if $n_{u,v} = n_{v,u}$. Thus, a connected graph is distance-balanced if and only if every pair of adjacent vertices is balanced. As there is no good reason why one would only be interested in the question of when every pair of adjacent vertices in a graph is balanced, the authors of~\cite{MikSpa18} (but see also~\cite{FreMik18}) proposed to study the following generalizations of the notion of distance-balancedness. A graph of diameter at least $\ell$ is said to be {\em $\ell$-distance-balanced} if each pair of its vertices at distance $\ell$ is balanced. Therefore, $1$-distance-balancedness is equivalent to being distance-balanced in the sense of~\cite{JKR08}. 

One could thus generalize the Mostar index in the sense that the {\em $\ell$-th Mostar index} of a connected finite graph $\G$ would be defined as
$$
	\mathrm{Mo}^\ell(\G) = \sum_{\{u,v\} \subseteq V(\G) \atop d(u,v) = \ell}|n_{u,v} - n_{v,u}|.
$$
However, to measure how much the graph is (un)balanced (in terms of distances) as a whole, one should really study the following parameter:
$$
	\mathrm{Mo}^*(\G) = \uB(\G) = \sum_{\{u,v\} \subseteq V(\G)} |n_{u,v} - n_{v,u}| = \sum_{\ell = 1}^D \mathrm{Mo}^\ell(\G),
$$
where $D$ is the diameter of $\G$. We call $\mathrm{Mo}^*(\G) = \uB(\G)$ the {\em distance-unbalancedness (index)} of $\G$. In~\cite{MikSpa18} the notion of high distance-balancedness was introduced where a graph is {\em highly distance-balanced} if each pair of vertices of the graph is balanced. In other words, $\G$ is highly distance-balanced if and only if $\uB(\G) = 0$. Therefore, the index $\uB(\G)$ is a measure of how ``unbalanced'' in terms of distances the graph $\G$ is. 

As with the Wiener index (which is defined as the sum of distances between all pairs of vertices of the graph), where this number divided by the number of all distinct pairs of vertices of the graph is also studied, it might also be interesting to study the {\em average distance-unbalancedness} of a graph. For a graph $\G$ of order $n$ this number is defined as 
$$
	\auB(\G) = \frac{\uB(\G)}{{n \choose 2}}.
$$

There are of course numerous interesting natural questions that arise. How close to $0$ can the (average) distance-unbalancedness of a graph that is not highly distance-balanced be? For each positive integer $n$ what is the smallest nonzero and the largest possible value of $\uB(\G)$ when $\G$ runs through the set of all connected graphs of order $n$? What if $\G$ runs just through the set of all members of order $n$ of a certain family of graphs (such as trees)?

Our paper is organized as follows. After some preliminaries in Section \ref{sec:Prelim} we compute the distance-unbalancedness index of some well-known families in Section \ref{sec:Fam}. We consider distance-unbalancedness of trees and general graphs in Sections \ref{sec:trees} and \ref{sec:min}, respectively. Along the way, a few conjectures are stated and some open problems are proposed.

%%%%%%%%%%%%%%%%%
\section{Preliminaries}
\label{sec:Prelim}

Throughout this paper all graphs are simple, undirected, finite and connected (unless otherwise specified). We first introduce some additional notation that will shorten the arguments regarding balancedness of pairs of vertices. For a graph $\G$ and its vertices $u, v$ we let $W_{u,v}(\G)$ be the set of all vertices of $\G$ that are closer to $u$ than to $v$, and we let $n_{u,v}(\G)$ be its cardinality (if the graph $\G$ is clear from the context we simply write $W_{u,v}$ and $n_{u,v}$). The {\em distance-unbalancedness} of the pair $\{u,v\}$ is then defined as $\uB_\G(u,v) = |n_{u,v}(\G) - n_{v,u}(\G)|$ (again, we usually omit the subscript $\G$). The Mostar index of a connected graph $\G$ is thus simply $\mathrm{Mo}(\G) = \sum_{\{u,v\} \in E(\G)}\uB(u,v)$ and the distance-unbalancedness index of $\G$ is $\uB(\G) = \sum_{\{u,v\} \subseteq V(\G)}\uB(u,v)$. 

For a positive integer $n$ the complete bipartite graph $K_{n,1}$ will be called the {\em star} of order $n+1$ and will be denoted by $S_n$. The graph obtained by taking a disjoint union of stars $S_n$ and $S_m$ and then adding an edge between the central vertex of $S_n$ (the one of valence $n$) and the central vertex of $S_m$ will be denoted by $\SS{n}{m}$ and called a {\em merged star}. The graph obtained from $\SS{n}{m}$ by subdividing the edge between the central vertices of $S_n$ and $S_m$ with one vertex (usually denoted by $x$) will be denoted by $\SSx{n}{m}$ and called a {\em subdivided merged star}.

The cycle (graph) of order $n \geq 3$ will be denoted by $C_n$ and the path of order $n$ (and thus length $n-1$) will be denoted by $P_n$. 

For an integer $k \geq 3$ and positive integers $n_1 \geq n_2 \geq \cdots \geq n_k$ we let the {\em spider graph} $\Sp(n_1,n_2,\ldots , n_k)$ be the tree of order $1+n_1+n_2+\cdots + n_k$ having one vertex of valence $k$ and all other vertices of valence at most $2$ such that the graph obtained by removing this unique vertex of valence $k$ is a disjoint union of $k$ paths of orders $n_1, n_2, \ldots , n_k$. For instance, for $n \geq 3$ the star $S_n$ is in fact the spider graph $\Sp(1,1,\ldots, 1)$ of order $n+1$.

%%%%%%%%%%%%%%%%%
\section{Distance-unbalancedness of some well-known families}
\label{sec:Fam}

As was pointed out in~\cite{MikSpa18} a sufficient condition for a pair of vertices $\{u,v\}$ of a graph $\G$ to be balanced is that there exists an automorphism of $\G$ interchanging $u$ and $v$. This simple but useful fact will be used throughout the paper. As an immediate corollary (see~\cite[Corollary~3.4]{MikSpa18}) each Cayley graph of an Abelian group is highly distance-balanced and thus its distance-unbalancedness index is $0$. 

In what follows we compute the distance-unbalancedness index $\uB(\G)$ for members of some well-known families of graphs. We first consider complete multipartite graphs.

\begin{proposition}
\label{pro:multi}
Let $n_1 \geq n_2 \geq n_3 \geq \cdots \geq n_k$ be positive integers. Then 
$$
	\uB(K_{n_1, n_2, \ldots , n_k}) = \sum_{1 \leq i < j \leq k}n_in_j(n_i - n_j).
$$
\end{proposition}

\begin{proof}
Let $\G = K_{n_1, n_2, \ldots , n_k}$. Partition its vertex-set into sets $V_1, V_2, \ldots , V_k$ with $|V_i| = n_i$, $1 \leq i  \leq k$, in such a way that for each $i$, $1 \leq i  \leq k$, each vertex of $V_i$ is adjacent to every vertex outside $V_i$ (and to no vertex from $V_i$). As any pair of vertices from $V_i$, $1 \leq i \leq k$, can be interchanged by an automorphism of $\G$, the only possible unbalanced pairs of vertices are pairs $\{u,v\}$, where $u \in V_i$ and $v \in V_j$ for some $1 \leq i < j \leq k$. Clearly, $n_{u,v} = |1+(n_j-1) - (1 + (n_i-1))| = n_i - n_j$, and so the result follows.
\end{proof}

As an immediate corollary of the above proposition we get the distance-unbalancedness index of stars. Namely, for any positive integer $n \geq 1$ we get that $\uB(S_n) = \uB(K_{n,1}) = n(n-1)$ and consequently $\auB(S_n) = \frac{2(n-1)}{n+1}$. The {\em wheel} graphs, which are related to them (the wheel graph $W_{n}$, where $n \geq 3$, is obtained from the star $S_n$ by connecting all of its leafs into a single cycle) have a slightly smaller distance-unbalancedness index. Namely, the following holds (the proof of the stated result is straightforward and is left to the reader). 

\begin{proposition}
For each $n \geq 3$ we have that $\uB(W_n) = n(n-3)$.
\end{proposition}

Even though we will have a closer look at distance-unbalancedness of trees in the next section we now compute this index for a very special kind of trees.

\begin{proposition}
\label{pro:paths}
Let $n$ be a positive integer. Then the following holds for the path $P_{n}$ of order $n$:
$$
	\uB(P_{n}) = \left\{\begin{array}{cc}
		\frac{(n-2)n(2n+1)}{12} & \text{if }  n\ \text{is even},\\
		\frac{(n-1)(n+1)(2n-3)}{12} &\text{if }  n\ \text{is odd},\end{array}\right. \quad \text{and}\quad
		\auB(\G) = \left\{\begin{array}{ccc}
		\frac{(n-2)(2n+1)}{6(n-1)} &\text{if }  n\ \text{is even},\\
		\frac{(n+1)(2n-3)}{6n} &\text{if }  n\ \text{is odd}.\end{array}\right.		
$$
\end{proposition}

\begin{proof}
Denote the vertices of $\G = P_{n}$ by $v_1, \ldots, v_n$, where $v_i \sim v_{i+1}$ for all $i$, $1 \leq i < n$. Let $i,j$ be integers with $1 \leq i < j \leq n$ and let $\G'$ be the subgraph of $\G$ induced on the set of all vertices $v_k$ with $i \leq k \leq j$. As there is an automorphism of $\G'$ interchanging $v_i$ and $v_j$, it follows that $\uB_{\G'}(v_i, v_j) = 0$, and so it clearly follows that $\uB_\G(v_i,v_j) = |(i-1)-(n-j)| = |n+1-i-j|$. Therefore,
$$
	\uB(\G) = \sum_{1 \leq i < j \leq n}|n+1-i-j| = \sum_{i = 1}^{n-1}\sum_{j=i+1}^n |n+1-i-j|.
$$
Let us denote this number by $a_n$. To compute $a_n$ one thus has to sum up all the numbers from the following upper triangular table (where the lines correspond to all possible values of $i$ from $1$ to $n-1$ in increasing order):
$$
	\begin{array}{ccccccccc}
	n-2 & n-3 & n-4 & n-5 & \cdots & 3 & 2 & 1 & 0 \\
	     & n-4 & n-5 & n-6 & \cdots & 2 & 1 & 0 & 1 \\
	     &      & n-6 & n-7 & \cdots & 1 & 0 & 1 & 2 \\
	     &      &      &      &  \vdots &   &   &   &  \\
	     &      &      &      &           & n-8 & n-7  & n-6  & n-5 \\
     	     &      &      &      &           &      & n-6  & n-5  & n-4 \\
     	     &      &      &      &           &      &       & n-4  & n-3 \\
     	     &      &      &      &           &      &       &       & n-2 \end{array}
$$
It is thus clear that $a_{n+2} = a_n + n(n+1)$ holds for all $n \geq 0$ where of course $a_1 = a_2 = 0$. Solving this elementary recursion yields the result of this proposition.
\end{proof}

We can use the above result to calculate the distance-unbalancedness index of some tube-like graphs (Cartesian products of paths by cycles). Note that each Cartesian product of two cycles is highly distance-balanced as it is a Cayley graph of an abelian group (the direct product of the two corresponding cyclic groups) and thus has distance-unbalancedness index $0$. This of course also holds for trivial tubes $P_1 \square C_m$ and $P_2 \square C_m$. 

We first fix some notation pertaining to tube-like graphs. Let $n \geq 1$ and $m \geq 3$ be integers and let $\G = P_n \square C_m$. The vertices of $\G$ will be denoted by pairs $(i,j)$, $i \in \{1,2,\ldots , n\}$, $j \in \ZZ_m$ in the natural way (so that for each $1 \leq i \leq n$ and each $j \in \ZZ_m$ the vertex $(i,j)$ is adjacent to $(i,j\pm 1)$, and for each $1 < i < n$ and each $j \in \ZZ_m$ the vertex $(i,j)$ is adjacent to $(i\pm 1, j)$). For each $i \in \{1,2,\ldots , n\}$ we let $\G^i$ be the subgraph of $\G$ induced on the set of vertices $\{(i,j) : j \in \ZZ_m\}$. We call $\G^i$ the {\em $i$-th column} of $\G$. Similarly, for each $j \in \ZZ_m$ the $j$-th {\em row} $\G_j$ of $\G$ is the subgraph of $\G$ induced on the set of vertices $\{(i,j) \colon 1 \leq i \leq n\}$. 

Since the permutations $\alpha, \beta, \gamma$, mapping each vertex $(i,j)$ to $(i,j+1)$, $(i,-j)$ and $(n-i+1,j)$, respectively (where the second coordinate is computed modulo $m$), are automorphisms of $\G$, it is clear that for any $j, j' \in \ZZ_m$ there is an automorphism of $\G$ interchanging the vertices $(1,j)$ and $(n,j')$, and so the corresponding pair of vertices is balanced in $\G$. A similar argument shows the following. Let $1 \leq i_1 < i_2 \leq n$ and let $j_1, j_2 \in \ZZ_m$. Then $(i_1,j_1)$ and $(i_2,j_2)$ are balanced in the subgraph of $\G$, induced by the union of columns $\G^{i_1} \cup \G^{i_1+1} \cup \cdots \cup \G^{i_2}$. This thus shows that the value of $\uB((i_1,j_1),(i_2,j_2))$ is completely determined by the vertices from $\G^1 \cup \G^2 \cup \cdots \cup \G^{i_1-1} \cup \G^{i_2+1} \cup \G^{i_2+2} \cup \cdots \cup \G^n$. This simple observation makes the computation of $\uB(\G)$ much easier. To illustrate how this can be done for each $n$ and $m$ we provide some details for the cases $m = 3$ and $m = 4$. For ease of reference we state the following elementary fact which will be used several times in the corresponding proofs:
\begin{equation}
\label{eq:sum}
	\sum_{i=1}^{n-1} |n-2i| = \left\{\begin{array}{cc}
		\frac{n(n-2)}{2} & \text{if } n\ \text{is even},\\
		\frac{(n-1)^2}{2} & \text{if } n\ \text{is odd}.\end{array}\right.	
\end{equation}

\begin{proposition}
\label{pro:PnC3}
Let $n$ be a positive integer. Then 
$$
	\uB(P_n \square C_3) = \left\{\begin{array}{cc}
		\frac{3n(n-2)(6n-1)}{4} & \text{if } n\ \text{is even},\\
		\frac{3(n-1)(2n+1)(3n-5)}{4} & \text{if } n\ \text{is odd}.\end{array}\right. 
$$
\end{proposition}

\begin{proof}
As mentioned in the discussion preceding the statement of the proposition the graphs $P_1 \square C_3 = C_3$ and $P_2 \square C_3$ are highly distance-balanced. Since the two expressions from the proposition for $n = 1$ and $n = 2$ indeed result in $0$, we can thus assume $n \geq 3$ for the rest of the proof. Denote $\G = P_n \square C_3$.

Let us determine the value $\uB((i_1,j_1),(i_2,j_2))$ for each pair of vertices where $i_1 \leq i_2$, $j_1, j_2 \in \ZZ_3$. In view of the above mentioned automorphisms $\alpha$ and $\beta$ we only need to consider the possibilities with $i_1 < i_2$. Moreover, by the argument from the paragraph preceding this proposition we only need to consider the vertices $(i,j)$ with $i < i_1$ or $i > i_2$. Now, for any such vertex $(i,j)$ its distance to $(i_1,j_1)$ is equal to $|i-i_1|$ or $|i-i_1|+1$, depending on whether $j = j_1$ or not, respectively. Similarly, the distance from $(i,j)$ to $(i_2,j_2)$ is equal to $|i-i_2|$ or $|i-i_2|+1$, depending on whether $j = j_2$ or not, respectively. Therefore, unless $i_2 = i_1+1$ and $j_1 \neq j_2$, we see that all the columns $\G^i$ with $i < i_1$ belong to $W_{(i_1,j_1),(i_2,j_2)}$ while all the columns $\G^i$ with $i > i_2$ belong to $W_{(i_2,j_2),(i_1,j_1)}$. In this case we thus get
$$
	\uB((i_1,j_1),(i_2,j_2)) = 3|(n-i_2) - (i_1-1)| = 3|n+1-i_1-i_2| = 3\cdot \uB_{P_n}(i_1,i_2).
$$
If however $i_2 = i_1+1$ and $j_1 \neq j_2$, then each vertex of the form $(i,j_2)$ with $i < i_1$, and similarly each vertex of the form $(i,j_1)$ with $i > i_2$, has equal distance to $(i_1,j_1)$ and $(i_2,j_2)$. In this case we thus get
$$
	\uB((i_1,j_1),(i_1+1,j_2)) = 2|n-2i_1|.
$$
It thus follows that $\uB(\G) = 3^3 \cdot \uB(P_n) - 3\cdot 2\cdot \sum_{i = 1}^{n-1}|n-2i|$, and so we can simply apply Proposition~\ref{pro:paths} and \eqref{eq:sum}.
\end{proof}

\begin{proposition}
\label{pro:PnC4}
Let $n$ be a positive integer. Then
$$
	\uB(P_n \square C_4) = \left\{\begin{array}{cc}
		\frac{2(n-2)(16n^2-13n+6)}{3} & \text{if } n\ \text{is even},\\
		\frac{2(n-1)(16n^2-29n+3)}{3} & \text{if } n\ \text{is odd}.\end{array}\right. 
$$
\end{proposition}

\begin{proof}
The idea of the proof is similar as before. Denote $\G =  P_n \square C_4$. We start with $4^3\cdot \uB(P_n)$ and then subtract an appropriate number to compensate for the pairs $(i_1,j_1), (i_2,j_2)$, $i_1 < i_2$, where it is not true that entire columns $\G^i$ with $i < i_1$ are contained in $W_{(i_1,j_1),(i_2,j_2)}$. Of course, this time such a situation can occur in different ways. One is for sure if $i_2 = i_1+1$ and $j_1 \neq j_2$ (there are two essentially different possibilities here), but there is also the possibility that $i_2 = i_1 + 2$ and $j_2 = j_1 + 2$. Let us consider each of these three possibilities in some detail. 

If $i_2 = i_1+1$ and $j_2 = j_1 \pm 1$, then for each $i < i_1$ the vertices of the form $(i,j)$ where $j \in \{j_2, j_1+2\}$ have equal distance to $(i_1,j_1)$ and $(i_2,j_2)$, while the remaining vertices of the form $(i,j)$ are in $W_{(i_1,j_1),(i_2,j_2)}$ (and similarly for the vertices $(i,j)$ with $i > i_2$). To compensate for these pairs we thus have to subtract $4\cdot 2 \cdot 2 \cdot \sum_{i=1}^{n-1}|n-2i|$ from $4^3\cdot \uB(P_n)$.

If $i_2 = i_1+1$ and $j_2 = j_1 + 2$, then for each $i < i_1$ the vertices of the form $(i,j_2)$ are in $W_{(i_2,j_2),(i_1,j_1)}$ while the remaining vertices of the form $(i,j)$ are in $W_{(i_1,j_1),(i_2,j_2)}$. To compensate for these pairs we thus have to subtract $4\cdot 1 \cdot 2 \cdot \sum_{i=1}^{n-1}|n-2i|$ from $4^3\cdot \uB(P_n)$.

Finally, if $i_2 = i_1 + 2$ and $j_2 = j_1 + 2$, then for each $i < i_1$ the vertices of the form $(i,j_2)$ have equal distance to $(i_1,j_1)$ and $(i_2,j_2)$ while the remaining vertices of the form $(i,j)$ are in $W_{(i_1,j_1),(i_2,j_2)}$. To compensate for these pairs we thus have to subtract $4\cdot 1 \cdot 1 \cdot \sum_{i=1}^{n-2}|(n-1)-2i|$ from $4^3\cdot \uB(P_n)$.

We can now apply Proposition~\ref{pro:paths} and \eqref{eq:sum} to obtain the stated result.
\end{proof}

Using a similar approach one can thus obtain the formula for $\uB(P_n \square C_m)$ for any particular choice of $m \geq 3$. We state just one more such result (without proof).

\begin{proposition}
\label{pro:PnC5}
Let $n$ be a positive integer. Then 
$$
	\uB(P_n \square C_5) = \left\{\begin{array}{cc}
		\frac{5(n-2)(50n^2-47n+24)}{12} & \text{if } n\ \text{is even},\\
		\frac{5(n-1)(50n^2-97n+21)}{12} & \text{if } n\ \text{is odd}.\end{array}\right. 
$$
\end{proposition}

%%%%%%%%%%%%%%%%%
\section{Distance-unbalancedness of trees}
\label{sec:trees}

In the previous section we determined the distance-unbalancedness index of members of a few well-known families of graphs. The stars $S_n$ and the paths $P_n$, which were among the studied families, are examples of trees. Observe that the distance-unbalancedness index of a star of order $n$ (which is $(n-1)(n-2)$) is much smaller than that of the path of the same order (which is roughly $n^3/6$). This is in stark contrast with the fact proved in~\cite{DMSTZ18} that the path of order $n$ has the smallest possible Mostar index among all trees of order $n$ while the star of order $n$ has the larges possible Mostar index among all trees of order $n$. This phenomenon can be explained to some extent by observing that the star of order $n$ has diameter $2$ and in fact its Mostar index coincides with its distance-unbalancedness, while on the other hand the path of order $n$ has diameter $n-1$, and its Mostar index forms just a tiny part of its distance-unbalancedness. 

It is the aim of this section to initiate an investigation of distance-unbalancedness of trees. It seems that determining this index for all trees of order $n$ is a very difficult (if not impossible) task. We are thus primarily interested in obtaining some preliminary results, making some computer-assisted investigations and identifying some interesting problems that might be feasible to solve in the future.

Note that for any tree $\G$ of order $n \geq 2$ and its leaf $u$ whose unique neighbor in $\G$ is $v$ we have that $\uB(u,v) = n-2$. This shows that the only highly distance-balanced trees are the trivial ones $K_1$ and $K_2$. For all of the remaining ones the distance-unbalancedness index is nonzero. One of the most fundamental problems to consider is thus for sure the following one:

\begin{problem}
\label{pro:treemin}
For each integer $n \geq 3$ determine the minimum and the maximum value of $\uB(\G)$ where $\G$ is a tree of order $n$ and classify all trees that attain these two values. 
\end{problem}

Using a computer one can of course determine the minimum and maximum value from the above problem at least for some small orders $n$. We investigated all trees up to order $15$ and obtained some interesting data. In the following table, for each order $n$, the smallest (min) and second smallest (min') possible values of distance-unbalancedness for trees of order $n$ are given, as well as the second largest (max') and largest (max) value of this index. We also state the number of all tress of given order ($\#\text{all}$), the number of trees attaining the smallest ($\#\text{min}$) and the number of trees attaining the largest ($\#\text{max}$) possible value of distance-unbalancedness.
$$
\begin{array}{c|ccccccccccccc}
	\text{n} & 		3 & 4 & 5   & 6   & 7  & 8 & 9 & 10 & 11 & 12 & 13 & 14 & 15\\
	\hline \hline
	\text{min} &  		2 & 6 & 12 & 20 &	30  & 42	& 56  &  72  &  90 &  110 & 132 & 156 & 182    \\
	\text{min'} &		  &    & 14 & 24 & 38  & 54 & 74  & 96  &  122 & 150 & 182 & 216 & 254 \\
	\text{max'} &         &    & 14 & 30 & 54  & 88 & 134 & 192 & 277 & 372 & 494 & 636 & 806 \\
	\text{max} &       2 & 6 & 16 & 32 &	56  & 90 & 138 &  198 & 278  & 378 & 495 & 638 & 808 \\ \hline
	\#\text{all} &		1 & 2 & 3   & 6  & 11  & 23	& 47  & 106 & 235  & 551 & 1301 & 3159 & 7741	\\
	\#\text{min} &     1 & 2 & 1   &  1 &	1    & 1  &   1  &  1   &  1 & 1  & 1 & 1 & 1 \\
	\#\text{max} & 1 & 2 & 1   & 	1 &	1   & 1   & 1    &  3  &  1 & 1 &  1 & 1 & 1 \\
\end{array}
$$
One can now analyze the trees attaining the smallest (few) and largest (few) values of distance-unbalancedness of given small orders. It turns out that for each $n$ with $5 \leq n \leq 15$ the unique tree $T$ of order $n$ with the smallest possible value of $\uB(T)$ among all trees of order $n$ is the star $S_n$ (with $\uB(S_n) = (n-1)(n-2)$). Moreover, it turns out that for each $n$ with $5 \leq n \leq 15$ there is a unique tree $T$ of order $n$ having the second smallest possible value of $\uB(T)$ among all trees of order $n$. This tree is isomorphic to the merged star $\SS{m}{m}$ with $n = 2m+2$ when $n$ is even, and to the subdivided merged star $\SSx{m}{m}$ with $n = 2m+3$ when $n$ is odd. Based on these observations we make the following conjectures.

\begin{conjecture}
\label{con1}
For any $n \geq 5$ the star $S_{n-1}$ is the unique tree of order $n$ having the smallest possible distance-unbalancedness index among all trees of order $n$.
\end{conjecture}

\begin{conjecture}
\label{con2}
For any $n \geq 5$, depending on whether $n$ is even or odd, respectively, the merged star $\SS{(n-2)/2}{(n-2)/2}$ or the subdivided merged star $\SSx{(n-3)/2}{(n-3)/2}$, respectively, is the unique tree of order $n$ having the second smallest possible distance-unbalancedness index among all trees of order $n$.
\end{conjecture}

It seems that proving the above two conjectures (especially the second one) will not be very easy. However, in the next few results we at least determine the distance-unbalancedness index of all merged stars and subdivided merged stars and show that for an even number $2n \geq 6$ the merged star $\SS{n-1}{n-1}$ is the unique tree among all merged stars and subdivided merged stars of order $2n$ that has the smallest possible distance-unbalancedness index. Similarly, we show that for an odd number $2n+1 \geq 5$ the subdivided merged star $\SSx{n-1}{n-1}$ is the unique tree among all merged stars and subdivided merged stars of order $2n$ that has the smallest possible distance-unbalancedness index.

\begin{proposition}
	\label{dstar:prop1}
	Let $n$ and $m$ be positive integers with $n \ge m \ge 1$. Then
	$$
	  \uB(\SS{n}{m})=(n+m)^2+(n-m)(nm+1)+2nm.
	$$
\end{proposition}
\begin{proof}
	Denote $\G = \SS{n}{m}$ and observe that since $n \ge m \ge 1$, the diameter of $\G$ is 3. It is now easy to see that 
	$$
	  \mathrm{Mo}^1(\G) = (n+m)^2+(n-m), \quad 
	\mathrm{Mo}^2(\G) = 2nm\quad \text{and}\quad 
	\mathrm{Mo}^3(\G) = nm(n-m),
	$$
and so the result follows.
\end{proof}

\begin{proposition}
	\label{dstar:prop2}
	Let $n,m$ be positive integers with $n \ge m \ge 1$. Then
	$$
	\uB(\SSx{n}{m})=(n+m)(n+m+1)+(n+1)|n-m-1|+(m+1)(3n-m+1)+nm(n-m).
	$$
\end{proposition}
\begin{proof}
Observe first that the diameter of $\G = \SSx{n}{m}$ is 4. It is easy to see that 
	$$
	\mathrm{Mo}^1(\G) = (n+m)(n+m+1)+|n-(m+1)| + n-m+1,
	$$
	$$
	\mathrm{Mo}^2(\G) = n(m+1) + m(n+1) + n-m,
	$$
$$
	\mathrm{Mo}^3(\G) = n|n-(m+1)| + m(n+1-m) \quad \text{and}\quad \mathrm{Mo}^4(\G) = nm(n-m),
$$
and so the result follows.
\end{proof}

\begin{theorem}
\label{thm:mstars}
Let $N \geq 3$ be an integer and let $\mathcal{S}_N$ be the set of all merged stars and all subdivided merged stars of order $N$. Then the following hold:
\begin{itemize}
\itemsep = 0pt
\item[(i)] If $N = 2(m+1)$ is even, then $\SS{m}{m}$ is the unique member of $\mathcal{S}_N$ having the smallest possible distance-unbalancedness index among all members of $\mathcal{S}_N$ and $\uB(\SS{m}{m}) = 6m^2$.
\item[(ii)] If $N = 2m+3$ is odd, then $\SSx{m}{m}$ is the unique member of $\mathcal{S}_N$ having the smallest possible distance-unbalancedness index among all members of $\mathcal{S}_N$ and $\uB(\SSx{m}{m}) = 6m^2+6m+2$.
\end{itemize}
\end{theorem}

\begin{proof}That the result holds for all $N$ with $3 \leq N \leq 15$ can easily be verified using Propositions~\ref{dstar:prop1} and~\ref{dstar:prop2}. For the rest of the proof we thus assume $N \geq 16$.

Consider first all merged stars $\SS{n}{m}$ where $N = n+m+2$ with $n \geq m \geq 1$. Of course, $m = N-n-2$ and $1 \leq m \leq (N-2)/2$ while $(N-2)/2 \leq n \leq N-3$. Write $f(n) = \uB(\SS{n}{m})$. Using Proposition~\ref{dstar:prop1} and rearranging terms we find that 
$$
	f(n) = -2n^3 + (3N-8)n^2 -(N^2-6N+6)n + (N-2)(N-3)
$$
is a polynomial in $n$ of degree $3$ with a negative leading coefficient. As $f(0) > 0$ and $f(2) = -N^2+19N-54 < 0$ (recall that $N \geq 16$), $f$ has a root in the interval $(0,2)$. Since $f((N-8)/2) = 30$ and $(N-8)/2 \geq 4 > 2$, $f$ has another root in the interval $(2,(N-8)/2)$. Finally, $f(N-2)=N^2-3N+2 > 0$ and $f(N-1)=6-4N < 0$, and so $f$ has its third root in the interval $(N-2,N-1)$. Since we are only interested in the value of $f(n)$ for $n$ satisfying $(N-2)/2 \leq n \leq N-3$, we thus only need to compare $f(N-3)$ to $f(n)$ for the smallest integer $n$ with $n \geq (N-2)/2$. Now,  
$$
	f(N-3) = 2(N-1)(N-3), \ f((N-2)/2) = \frac{3(N-2)^2}{2}\ \text{and}\ f((N-1)/2) = \frac{7N^2-28N+29}{4},
$$
and so $f(N-3) > f((N-1)/2) > f((N-2)/2)$, implying that for $N$ even $\uB(\SS{n}{m})$ is minimal for $m = n = (N-2)/2$ (and equals $3(N-2)^2/2 = 6n^2$), while for $N$ odd $\uB(\SS{n}{m})$ is minimal for $n = (N-1)/2$ and $m = (N-3)/2$ (and equals $(7N^2-28N+29)/4$).

Consider now all subdivided merged stars $\SSx{m}{n}$ where $N = n+m+3$ with $n \geq m \geq 1$. Then $m = N-n-3$, $1 \leq m \leq (N-3)/2$ and $(N-3)/2 \leq n \leq N-4$. Write $g(n) = \uB(\SSx{n}{m})$. If $m = n$ (which is only possible if $N = 2n+3$ is odd), then Proposition~\ref{dstar:prop2} implies that 
$$
	g(n) = 6n^2+6n+2 = \frac{3N^2-12N+13}{2}.
$$
If however $n \geq m+1$ (in which case $n \geq (N-2)/2$), then Proposition~\ref{dstar:prop2} implies that 
$$
	g(n) = -2n^3 + (3N-11)n^2 - (N^2-10N+17)n
$$
is a polynomial in $n$ of degree $3$ with a negative leading coefficient. Note that $g(0) = 0$. Similarly as above we find that $g$ also has one root in the interval $((N-10)/2,(N-8)/2)$ and another one in the interval $(N-2,N-1)$. Therefore, in the interval $((N-2)/2, N-4)$ the function $g$ attains its minimum either in $N-4$ or in the smallest integer $n$ with $n \geq (N-2)/2$. As
$$
	g(N-4) = (N-4)(3N-5),\ g((N-2)/2) = \frac{7N^2-30N+32}{4}\ \text{and}\ g((N-1)/2) = 2(N-1)(N-3),
$$
we see that $g(N-4) > g((N-1)/2) > g((N-2)/2)$ (recall that $N \geq 16$), and so the fact that also $(3N^2-12N+13)/2 < 2(N-1)(N-3)$ implies that for $N$ odd $\uB(\SSx{n}{m})$ is minimal for $m = n = (N-3)/2$ (and equals $(3N^2-12N+13)/2 = 6n^2+6n+2$), while for $N$ even $\uB(\SS{n}{m})$ is minimal for $n = (N-2)/2$ and $m = (N-4)/2$ (and equals $(7N^2-30N+32)/4$).

To complete the proof we now simply have to compare the above obtained smallest possible value of $\uB(\SS{n}{m})$ and $\uB(\SSx{n}{m})$ depending on whether $N$ is even or odd. 
\end{proof}

Returning to the above data regarding the distance-unbalancedness index of trees up to order $15$, it seems that determining the largest (or second largest) possible value of this index among all trees of order $n$ is even more difficult than proving Conjecture~\ref{con1}. Inspecting the trees attaining these values for $n$ up to $15$ one finds that all of these trees are spider graphs. As was indicated in the above table, for each integer $n$ with $5 \leq n \leq 15$, except for $n = 10$, there is a unique tree of order $n$ having the largest possible distance-unbalancedness index among all trees of order $n$. For $n = 10$ there are three such trees, namely the spiders $\Sp(3,2,2,2), \Sp(3,3,2,1)$ and $\Sp(3,2,2,1,1)$. For $n$ with $5 \leq n \leq 15$, $n \neq 10$, the unique maximal tree with respect to the distance-unbalancedness index is given in the following table, where for $\Sp(n_1, n_2, \ldots , n_k)$ we simply state the sequence $(n_1,n_2, \ldots , n_k)$ (see also Figure~\ref{fig:spiders}).

{\small 
\begin{tabular}{c|c@{\,}c@{\,}c@{\,}c@{\,}c@{\,}c@{\,}c@{\,}c@{\,}c@{\,}c}
	n & 		5   & 6   & 7  & 8 & 9 & 11 & 12 & 13 & 14 & 15 \\
	\hline 
	max &  (2,1,1) & (2,2,1) & (2,2,1,1) & (2,2,2,1) & (3,2,2,1) & (3,3,2,2) & (3,3,2,2,1) & (3,3,3,2,1) & (3,3,3,2,2) & (4,3,3,2,2) 
\end{tabular}
}
\begin{figure}[!h]
\begin{center}
	\includegraphics[scale=0.9]{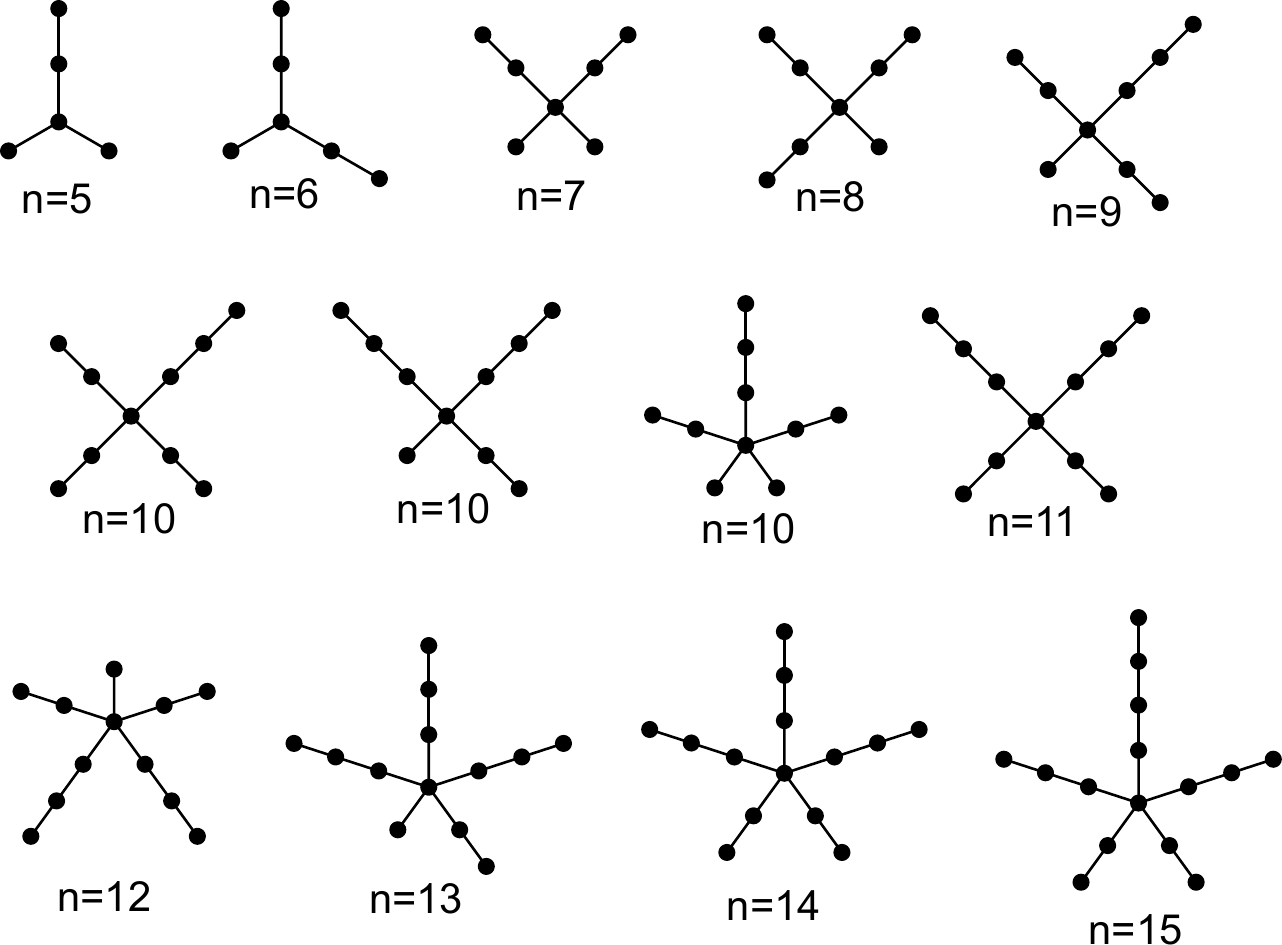}
	\caption{Maximal trees with respect to the distance-unbalancedness index.}
	\label{fig:spiders}
\end{center}
\end{figure}
\medskip

In view of this data it is interesting to ask whether for all $n \geq 11$ there is a unique tree with maximal distance-unbalancedness index among all trees of order $n$ or there are some values of $n$ where a similar situation as for $n = 10$ occurs. Also, the following problem is of interest.

\begin{problem}
\label{prob:spiders}
For each integer $n \geq 5$ classify the spiders of order $n$ that have the largest possible distance-unbalancedness index among all spiders of order $n$.
\end{problem}

%%%%%%%%%%%%%%%%%
\section{Minimal distance-unbalancedness}
\label{sec:min}

In the previous section we considered the problem of determining the smallest nonzero and the largest possible distance-unbalancedness index among all trees of given order. Given that those two problems appear to be rather difficult it seems that determining the smallest nonzero (and the largest) possible distance-unbalancedness index among all connected graphs of given order might be even more difficult. It seems natural to speculate that, as with trees, the smallest possible nonzero distance-unbalancedness index should increase with the order of the considered graphs. Nevertheless, as we shall see later, this is not true.

One can of course again use computers to investigate all connected graphs up to a given reasonably small order. We have looked at all connected graphs up to order $10$ using the lists of all such graphs available at~\cite{McKayWeb}. For each order $n$, $3 \leq n \leq 10$, we determined the smallest possible nonzero value $\uB(\G)$ among all connected graphs of order $n$ and determined all graphs $\G$ of order $n$ attaining this value. The following table summarizes the obtained data where for each order $n$ the smallest possible nonzero distance-unbalancedness index (min) for connected graphs of order $n$ and the number of graphs attaining this value ($\#\text{min}$) are given.
$$
\begin{array}{c|cccccccc}
	\text{n} & 		3 & 4 & 5   & 6   & 7  & 8 & 9 & 10 \\
	\hline 
	\text{min} &  		2 & 4 & 4 & 8   &	6  & 8  & 6  & 8 \\
	\#\text{min} &     1 & 1 & 2 & 10  &	8  & 2  & 1  & 17 \\
\end{array}
$$

With trees the computer generated data is large enough to give a very strong indication of what should be proved (Conjecture~\ref{con1}). In this general setting the set of orders considered is perhaps a bit to small to draw conclusions but it does indicate that perhaps the above mentioned ``natural speculation'' is not correct. Nevertheless, we propose the following problem (although the obtained data suggests that at least the second part of the problem might be extremely difficult if not impossible to solve). 

\begin{problem}
\label{prob:allmin}
For each integer $n \geq 3$ determine the smallest possible nonzero value of the distance-unbalancedness index among all connected graphs of order $n$ and classify all graphs attaining this value.
\end{problem}

Looking at the values from the above table it seems that perhaps the solution to the first part of the problem is that for $n \geq 6$ the minimal nonzero value of the distance-unbalancedness index among all connected graphs of order $n$ is $6$ or $8$, depending on whether $n$ is odd or even, respectively. Of course, the obtained data is much too small to justify this speculation. Trying to determine whether it is indeed plausible that there does exist a constant $C$ such that for each integer $n$ the smallest nonzero distance-unbalancedness index among all connected graphs of order $n$ is at most $C$, it makes sense to have a closer look at the minimal graphs obtained for orders up to $n = 10$ (those of order $n$ that have the minimal possible nonzero distance-unbalancedness index among all connected graphs of order $n$). 

Doing so we discovered two graphs that fall into a known infinite family. Namely, the unique minimal graph of order $9$ and one of the 10 minimal graphs of order $6$ belong to an infinite family of bi-regular graphs (some vertices are of valence $3$ while others are of valence $4$) constructed in~\cite{IlKlMi10} (see Figure~\ref{fig:kites} where these two graphs and the members of the family of orders $12$ and $15$ are illustrated). This suggests that determining the distance-unbalancedness index of all members of this infinite family might be of interest. Interestingly enough, the graphs from the family are of two very different types regarding $\ell$-distance-balancedness. Those of even order are not $1$-distance-balanced but are $\ell$-distance-balanced for all $\ell \geq 2$, while those of odd order are $1$-distance-balanced but are not $\ell$-distance-balanced for any odd $\ell \geq 3$.
\begin{figure}[!h]
\begin{center}
	\includegraphics[scale=0.75]{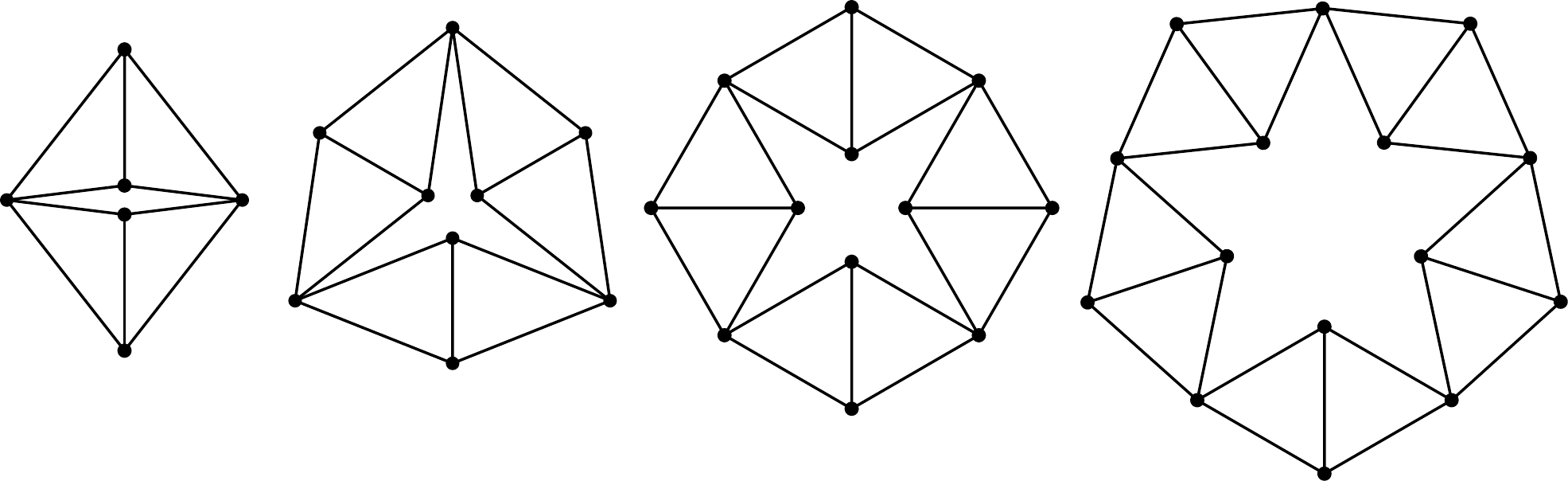}
	\caption{The kite graphs $\Ki(n)$ for $2 \leq n \leq 5$.}
	\label{fig:kites}
\end{center}
\end{figure}

\begin{construction}
For each integer $n \geq 2$ the {\em Kite graph} $\Ki(n)$ is the graph of order $3n$ with vertex set 
$$
	\{a_i \colon i \in \ZZ_n\} \cup \{b_i \colon i \in \ZZ_n\} \cup \{c_i \colon i \in \ZZ_n\},
$$
where for each $i \in \ZZ_n$ the vertices $b_i$ and $c_i$ are adjacent to each other and to $a_i$ and $a_{i+1}$ (the indices are computed modulo $n$). The reason we chose this name for them is that we can think of obtaining $\Ki(n)$ by taking $n$ copies of the $4$-cycle with one additional chord (which can be viewed as a kite) and then gluing them along vertices of valence $2$ in a cyclic fashion (so that the ``old'' vertices of valence $2$ become of valence $4$ in the obtained graph).
\end{construction}

\begin{proposition}
\label{pro:zmaji}
Let $n \geq 2$ be an integer. Then the graph $\Ki(n)$ is a bi-regular graph of diameter $n$ with $n$ vertices of valence $4$ and $2n$ vertices of valence $3$. Moreover, the following hold:
\begin{itemize}
\itemsep = 0pt
\item[(i)] If $n$ is even, then $\Ki(n)$ is not $1$-distance-balanced but is $\ell$-distance-balanced for all $\ell$, $2 \leq \ell \leq n$. Moreover, $\uB(\Ki(n)) = 4n$.
\item[(ii)] If $n$ is odd, then $\Ki(n)$ is $\ell$-distance-balanced for $\ell=1$ and all even $\ell$ with $\ell \leq n-1$, but is not $\ell$-distance-balanced for any odd $\ell$, $3 \leq \ell \leq n$. Moreover, $\uB(\Ki(n)) = 2n(n-2)$.
\end{itemize}
\end{proposition}

\begin{proof}
The initial assertions follow straight from the definition. Let $\rho$ and $\sigma$ be the permutations of the vertex-set of $\Ki(n)$ such that for each $i \in \ZZ_n$
$$
	a_i\rho = a_{i+1}, b_i\rho = b_{i+1}, c_i\rho = c_{i+1},\ \text{and}\ a_i\sigma = a_{-i}, b_i\sigma = b_{-(i+1)}, c_i\sigma = c_{-(i+1)}.
$$
Furthermore, le $\tau_0$ be the permutation of the vertex-set of $\Ki(n)$, interchanging $b_0$ with $c_0$ and leaving all other vertices fixed. It is easy to see that these three permutations are in fact automorphisms of $\Ki(n)$. The automorphism group of $\Ki(n)$ thus contains the group $\la \rho, \sigma, \tau_0\ra$, and so it is clear that the only possible unbalanced pairs are of the form $\{a_i, b_j\}$ and $\{a_i, c_j\}$. Moreover, it suffices to determine all $\uB(a_0, b_j)$ where $j < n/2$. To do this observe that $c_0$ is the unique vertex which has equal distance to $a_0$ and $b_0$ while for $1 \leq j < n/2$ no vertex of $\Ki(n)$ has equal distance to $a_0$ and $b_j$. It is now straightforward to verify that if $n$ is even $\uB(a_0,b_0) = 1$ and $\uB(a_0,b_j) = 0$ for all $j$, $1 \leq j < n/2$, while if $n$ is odd $\uB(a_0,b_0) = 0$ (see also~\cite{IlKlMi10}) and $\uB(a_0, b_j) = 1$ for all $j$, $1 \leq j < n/2$. To complete the proof one now simply has to count the number of all such pairs. We leave details to the reader.
\end{proof}

The above proposition shows that the graph $\Ki(n)$, where $n$ is odd, has a rather large distance-unbalancedness index, and so it is not surprising that there exist graphs of order $3n$ with a smaller nonzero distance-unbalancedness index (except of course for $n = 3$). Indeed, for any odd $n \geq 5$, Proposition~\ref{pro:multi} and Proposition~\ref{pro:zmaji} show that $\uB(\Ki(n)) = 2n(n-2) \geq 6n > 6n-4 = \uB(\G)$, where $\G$ is the complete multipartite graph of order $3n$ of the form $K_{2,1,1,\ldots , 1}$. 

While on the other hand the graphs $\Ki(n)$ with $n$ even seem to be better in this respect their distance-unbalancedness index is still not bounded by a constant. But there is a different generalization of the graph $\Ki(2)$ (which is minimal among connected graphs of order $6$ in the above mentioned sense) that gives such a constant at least for graphs of even orders. Observe that (but see also~\cite{IlKlMi10}) the construction of the graph $\Ki(n)$ can also be described as follows. Start with the cycle $C_{2n}$ of length $2n$ and think of its vertices being labeled by elements of $\ZZ_{2n}$ in a natural way. For each even $i \in \ZZ_{2n}$ introduce a copy $i'$ of the vertex $i$ and connect it to each of $i$, $i-1$ and $i+1$. The resulting graph is of course $\Ki(n)$. But what if we introduce copies of just a few vertices and not of every other one in $C_{2n}$?

\begin{construction}
Let $n \geq 2$ be an integer. Label the vertices of the cycle $C_{2n}$ in a natural way by elements of $\ZZ_{2n}$ so that for each $i$ the vertex $i$ is adjacent to $i+1$ and $i-1$. Introduce two new vertices, say $x$ and $y$, and connect $x$ to $0$, $-1$ and $1$, and $y$ to $n$, $n-1$ and $n+1$. Denote the obtained graph of order $2n+2$ by $\tilde{C}_{2n}$ (see Figure~\ref{fig:cycles2} where the first few examples are illustrated).
\end{construction}
\begin{figure}[!h]
\begin{center}
	\includegraphics[scale=0.75]{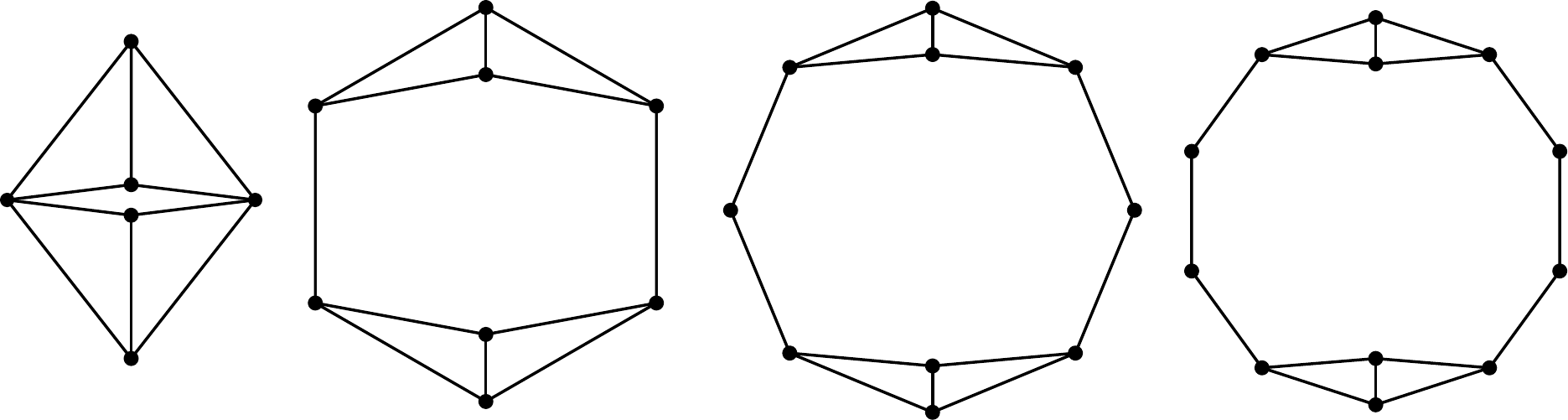}
	\caption{The graphs $\tilde{C}_{2n}$ for $2 \leq n \leq 5$.}
	\label{fig:cycles2}
\end{center}
\end{figure}

\begin{proposition}
\label{pro:CQ}
Let $n \geq 2$ be an integer. Then $\uB(\tilde{C}_{2n}) = 8$. Moreover, the graph $\tilde{C}_{2n}$ is of order $2n+2$ and diameter $n$ and is not $1$-distance-balanced, but is $\ell$-distance-balanced for all $\ell$, $2 \leq \ell \leq n$.
\end{proposition}

\begin{proof}
Let $\tau_x, \tau_y$ be the automorphisms of $\tilde{C}_{2n}$ fixing all of its vertices except $0$ and $x$, respectively $n$ and $y$, which are swapped. Similarly, let $\tau_0$ be the automorphism interchanging vertices $i$ and $-i$ of $\tilde{C}_{2n}$ for all $i \in \ZZ_{2n}$ and fixing each of $0$, $n$, $x$ and $y$. Moreover, let $\rho$ be the automorphism interchanging each $i \in \ZZ_{2n}$ with $n+i$ and interchanging $x$ with $y$. In view of these automorphisms it thus suffices to consider only pairs of vertices of the form $\{i,j\}$, where $0 \leq i < j \leq n-1$. In fact, because of the automorphism $\tau_0\rho$ we can in fact assume that $i < n-j$ (if $i = n-j$ this automorphism actually swaps $i$ and $j$, and so this pair is balanced). Now, if $i \neq 0$ or $j \neq 1$, then $x \in W_{i,j}$ while $y \in W_{j,i}$, and so the pair $\{i,j\}$ is clearly balanced. On the other hand, as $x$ is a common neighbor of $0$ and $1$, the pair $\{0,1\}$ is not balanced and $\uB(0,1) = 1$ (as $y \in W_{1,0}$). Therefore, the only unbalanced pairs correspond to the 8 edges $\{0, \pm 1\}$, $\{x, \pm 1\}$, $\{n, n \pm 1\}$, $\{y, n \pm 1\}$ and so the result follows.
\end{proof}

Perhaps there is a similar construction (based on a suitable cycle) that for each odd $n \geq 7$ results in a graph of order $n$ with distance-unbalancedness index equal to $6$, but it is not completely obvious how to obtain it. We thus leave the question of whether (and how) such graphs might be constructed unanswered. 
\bigskip

We conclude the paper with the following remarks. Even though we were able to obtain a connected graph $\G$ of even order $n$ with $\uB(\G) = 8$ for each even $n \geq 6$, it does feel that in obtaining these graphs we were ``cheating'' a bit - the examples were obtained by taking a cycle (which is in fact distance transitive, implying that for any pair of vertices there is an automorphism swapping them) and then changing it very slightly so as to introduce a tiny unbalancedness. But what if one only considers regular graphs? It is well known and easy to see that all (connected) regular graphs up to order $6$ are Cayley graphs of abelian groups and are thus highly distance-balanced. The only connected regular graph of order $7$ which is not a Cayley graph (of an abelian group) is the complement of the disjoint union of a $3$-cycle and a $4$-cycle. It is not difficult to see however, that this graph is highly distance-balanced. 

With order $8$ the situation changes. In fact, the two connected graphs of order $8$ with minimal nonzero distance-unbalancedness index (recall that it is equal to $8$) are both cubic. One of them is $\tilde{C}_6$, while the other is obtained from the $8$-cycle by adding the edges $\{0,4\}, \{1,3\}, \{2,6\}$ and $\{5,7\}$. Using a computer one can verify that all connected regular graphs of order $9$ are again highly distance-balanced. For order $10$ recall that the minimal nonzero distance-unbalancedness index among all connected graphs of order $10$ is $8$ and that it is attained by $17$ graphs. It turns out that apart from $\tilde{C}_8$ (which is not regular) all of the remaining $16$ graphs are $4$-regular. Finally, using a computer one can verify that the smallest nonzero distance-unbalancedness index among all connected regular graphs of order $11$ is $10$ and that it is attained by $43$ graphs. 

It is not difficult to see that $9$ is the largest odd order $n$ such that there does not exist a connected regular graph of order $n$ which is not highly distance-balanced. Namely, let $n = 2m+1$ with $m \geq 5$. Take any connected $4$-valent graph $\Delta$ of order $m$ and let $x$ and $y$ be a pair of adjacent vertices in $\Delta$. Construct a new graph $\G$ by taking two copies of $\Delta$ (denote the second copy by $\Delta'$ and the analogues of $x$ and $y$ in $\Delta'$ by $x'$ and $y'$), deleting the edges $xy$ and $x'y'$, introducing a new vertex $z$ and connecting it to $x, y, x'$ and $y'$. The obtained graph is clearly a connected $4$-regular graph of order $n$, but the pair $\{z,x\}$ is not balanced as $W_{z,x}$ contains at least $z$, $y$ and the whole copy $\Delta'$. In a similar way a connected cubic graph of order $2n$ for each $n \geq 5$ which is not highly distance-balanced can be constructed (this time start with two copies of a cubic graph of order $n-1$ and then add a pair of connected vertices in a similar way as before, but see also~\cite[Theorem~5.7]{MikSpa18}). The following problem thus might be of interest.

\begin{problem}
\label{prob:reg}
For each positive integer $n \geq 10$ determine the smallest possible nonzero value of the distance-unbalancedness index among all connected regular graphs of order $n$ and classify all graphs attaining this value.
\end{problem}

\section*{Acknowledgements}

\noindent
\v S. Miklavi\v c acknowledges support by the Slovenian Research Agency (research program P1-0285 and research projects N1-0062, J1-9110, J1-1695, N1-0140, N1-0159, J1-2451).

\noindent
P.~\v Sparl acknowledges support by the Slovenian Research Agency (research program P1-0285 and research projects J1-9108, J1-9110, J1-1694, J1-1695, J1-2451).


\begin{thebibliography}{}
	

\bibitem{BDHKS} S. Bessy, F. Dross, K. Hrinakova, M. Knor, R. \v{S}krekovski, Maximal Wiener index for graphs with prescribed number of blocks, {\em Appl. Math. Comput.}, {\bf 380} (2020), 1--7. 

\bibitem{AD} A. A. Dobrynin, The Szeged and Wiener Indices of Line Graphs,  {\em MATCH Commun. Math. Comput. Chem.} {\bf 79} (2018), 743--756.

\bibitem{DMSTZ18} T.~Do\v sli\'c, I.~Martinjak, R.~\v Skrekovski, S.~Tipuri\'c Spu\v zevi\'c, I.~Zubac,
		Mostar index,
		{\em J. Math. Chem.} {\bf 56} (2018), 2995--3013.
		
\bibitem{FreMik18} B.~Frelih, \v S.~Miklavi\v c,
		On $2$-distance-balanced graphs,
		{\em Ars Math. Contemp.} {\bf 15} (2018), 81--95.
	
\bibitem{GFK} I. Gutman, B. Furtula, V. Katani\'{c}, Randi\'{c} index and information, {\em AKCE Int. J. Graphs Comb.} {\bf 15} (2018), 307--312.

\bibitem{Han99} K.~Handa, 
		Bipartite graphs with balanced $(a,b)$-partitions,
             {\em Ars Combin.} {\bf 51} (1999), 113--119.
             
\bibitem{IlKlMi10} A.~Ili\'c, S.~Klav\v{z}ar, M.~Milanovi\'c, 
		On distance-balanced graphs,
             {\em European J. Combin.} {\bf 31} (2010), 733--737.
             
\bibitem{JKR08} J.~Jerebic, S.~Klav\v zar, D.~F.~Rall, 
		Distance-balanced graphs,  
		{\em Ann. Combin.} {\bf 12} (2008), 71--79.
		
\bibitem{KST} M. Knor, R. \v{S}krekovski, A. Tepeh, Mathematical Aspects of Balaban Index, {\em MATCH Commun. Math. Comput. Chem.} {\bf 79} (2018),  685--716.

\bibitem{KST1} M. Knor, R. \v{S}krekovski, A. Tepeh, Trees with the maximal value of Graovac-Pisanski index, {\em Appl. Math. Comput.} {\bf 358} (2019), 287--292.
		
\bibitem{McKayWeb} B.~McKay,
		The list of all simple graphs of order at most $10$, 
		available at {\texttt http://users.cecs.anu.edu.au/\~{}bdm/data/graphs.html}

\bibitem{MikSpa18} \v S.~Miklavi\v c, P.~\v Sparl,
		$\ell$-distance-balanced graphs,
		{\em Discrete Appl. Math.} {\bf 244} (2018), 143--154. 
		
\bibitem{WB} J. F. Wang, F. Belardo, A Lower Bound for the First Zagreb Index and Its Application, {\em MATCH Commun. Math. Comput. Chem.} {\bf 74} (2015), 35--56.
		
\end{thebibliography}
\end{document}